\documentclass[10pt]{amsart}
\usepackage{amsfonts,amssymb,amscd,amsmath,enumerate,verbatim,calc}
\usepackage{color,soul}
\usepackage{mathtools}
\usepackage{graphicx}
\everymath{\displaystyle}
\usepackage{fancyhdr}
\usepackage{tikz}
\usetikzlibrary{shapes.geometric, arrows}
\tikzstyle{arrow} = [thick,->,>=stealth]
\tikzstyle{process} = [rectangle, minimum width=4cm, minimum height=2cm, text centered, text width=2.5cm, draw=green, fill=green!10]
\newcommand{\Hom}{\text{Hom}}
\newcommand{\ran}{\text{Im }}
\newcommand{\rank}{\text{rank }}

\newcommand{\IFF}{\text{if and only if}}

\newcommand{\R}{\mathbb{R} }
\newcommand{\C}{\mathbb{C} }

\theoremstyle{plain}

\newtheorem{theorem}{Theorem}[section]
\newtheorem{corollary}[theorem]{Corollary}

\newtheorem{proposition}[theorem]{Proposition}

\theoremstyle{definition}
\newtheorem{definition}[theorem]{Definition}

\newtheorem{remark}[theorem]{Remark}
\newtheorem{example}[theorem]{Example}

\begin{document}
\title{Matrix representations of linear transformations on bicomplex space}
\author{Anjali}
\email{2019vsasmaf01@gbu.ac.in}
\address{Department of Applied Mathematics, Gautam Buddha University, Greater Noida, Uttar Pradesh 201312, India}

\author{Fahed~Zulfeqarr}
\email{fahed@gbu.ac.in}
\address{Department of Applied Mathematics, Gautam Buddha University, Greater Noida, Uttar Pradesh 201312, India}

\author{Akhil Prakash}
\email{gm5537@myamu.ac.in}
\address{Department of Mathematics, Aligarh Muslim University, Aligarh, Uttar Pradesh 202002, India}
\author{Prabhat Kumar}
\email{prabhatphilosopher@gmail.com}
\address{Department of Applied Mathematics, Gautam Buddha University, Greater Noida, Uttar Pradesh 201312, India}

\subjclass[IMS]{Primary 15A04, 15A30; Secondary 30G35}
\date{\today}
	
\begin{abstract}
An algebraic investigation on bicomplex numbers is carried out here. Particularly matrices and linear maps defined on them are discussed. A new kind of cartesian product, referred to as an idempotent product, is introduced and studied. The elements of this space are linear maps of a special form. These linear maps are examined with respect to usual notions like kernel, range, and singularity. Their matrix representation is also discussed. 
\end{abstract}
\maketitle

\section*{Introduction} 
The theory of bicomplex numbers has been a thrust area of current research in mathematics. It has evolved a lot in the recent past. Several researchers (see \cite{alpay2023interpolation, alpay2014basics,  cerroni2017theory, futagawa1928, futagawa1932, riley1953, ringleb1933,  rochon2004, srivastava2003, srivastava2008}) have contributed a lot to the field. They have been working in different directions to analyze their properties and to create concepts consistent with a unified approach to the multivariate theory of complex numbers. In this process, some concepts like bicomplex topology \cite {srivastava2011}, differentiability and analyticity of bicomplex function \cite{price2018}, power series,  bicomplex matrices, bicomplex Riemann zeta function and  Dirichlet series, integrability, Cauchy’s theory, etc., have been established with necessary modifications. Also, we have followed some of the results and symbols from the existing literature on the fundamental concepts from the book \cite{price2018} on bicomplex functions.

In section 1, bicomplex numbers and some basic algebraic structures built on them are introduced and discussed. The section also stresses the role of the idempotent representation of bicomplex numbers in their study. A new kind of cartesian product, known as an idempotent product, is defined and discussed. The space becomes a central object of investigation. In section 2, we have revised bicomplex complex linear operator and also we have defined bicomplex linear transformation in this section.

Section 3 deals with matrices and linear maps defined on bicomplex space. It connects bicomplex matrices with elements of an idempotent product. The last section focuses on how the elements of the idempotent product behave for kernel, range, invertibility, and singularity. It renders some results of their matrix representation.

\section{Preliminaries and Notations} 
This section introduces bicomplex numbers. It deals with basic notions such as idempotent representation and the cartesian product of bicomplex space. It presents some fundamental results on bicomplex numbers.

\noindent {\bf Bicomplex numbers:}
A bicomplex number is an element of the form : 
\[
\xi = {u}_{1}  + {i}_{1} {u}_{2}+ {i}_{2} {u}_{3} + {i}_{1} {i}_{2} {u}_{4},\; \;{u}_{k}\in \R, \;\;1\leqslant k\leqslant 4, \;\; \mbox{with}\;\; \;{i}_{1} {i}_{2} \ = \ {i}_{2} {i}_{1},\; \; {i}_{1}^2 = {i}_{2}^2  =  -1. 
\]
The set of all bicomplex numbers is denoted by $\C_{2}$ and is referred to as the bicomplex space. The symbols $\mathbb{C}_{1},\; \mathbb{C}_{0}$, for convenience, denote the set of all complex numbers and the set of all real numbers respectively. The bicomplex space $\C_2$ can be described in the following two ways:
\begin{eqnarray*}
\mathbb{C}_{2} &\coloneqq & \{{u}_{1}  + {i}_{1} {u}_{2}  + {i}_{2} {u}_{3}  + {i}_{1} {i}_{2} {u}_{4}\;:\; {u}_{1} , {u}_{2} , {u}_{3} , {u}_{4} \in\mathbb{C}_{0}\}, and \\
 \mathbb{C}_{2} &\coloneqq & \{{z}_{1} + {i}_{2} {z}_{2}\;:\;  {z}_{1}, {z}_{2} \in \mathbb{C}_{1}\}.
\end{eqnarray*}
 $\C_2$ contains zero-divisors and hence it is not a field but an algebra over $\C_1$. There are exactly four idempotent elements in $\C_2$, viz., $0,1,e_1,e_2$ where $e_1,e_2$ represent two nontrivial idempotent elements defined as follows:
 \[ e_{1} \coloneqq \frac{(1 + i_{1} i_{2})}{2}\quad  \mbox{and} \quad e_{2} \coloneqq \frac{(1 - i_{1} i_{2})}{2}.\]
Notice that \;$e_{1} + e_{2} =1$\; and \;$e_{1}e_{2}  = e_{2}e_{1} = 0$. Linear independence of these elements over $\C_1$ gives rise to a new representation of all bicomplex numbers known as idempotent representation.

\noindent{\bf Idempotent representation of bicomplex numbers:}
Every bicomplex number $\xi = {z}_{1}+ {i}_{2} {z}_{2}$ can  be uniquely represented as the complex combination of elements $e_{1}$ and $e_{2}$  in the following form
\[
\xi =  (z_{1} -  i_{1} z_{2}) e_{1} + (z_{1}  +  i_{1} z_{2}) e_{2},
\]
where complex numbers $(z_{1} -  i_{1} z_{2})$ and $(z_{1} +  i_{1} z_{2})$ are called idempotent components of $\xi$ and will be denoted by $\xi^-$ and $\xi^+$ respectively. Therefore the number $\xi = {z}_{1}+ {i}_{2} {z}_{2}$ can also be written as $\xi = \xi^- e_{1} + \xi^+ e_{2}$, where $\xi^-  = z_{1} - i_{1} z_{2}$ and  $\xi^+ = z_{1}  + i_{1} z_{2}$. 
\begin{remark}
The idempotent representation of product of elements $\xi, \eta\in\C_2$ can be seen easily to be as:
\[
\xi\cdot\eta = \left(\xi^- \eta^-\right)e_{1} + \left(\xi^+ \eta^+\right)e_{2}.
\]
\end{remark}

\noindent{\bf Cartesian product:}
The $n$-times cartesian product of $\C_2$ is represented by $\C_2^n$ and consisting of all $n$-tuples of bicomplex numbers of the form \;$(\xi_{1}, \xi_{2},\ldots,\xi_{n})$,\;where $\xi_{i} \in \C_2 ;\; i = 1,2,\ldots,n$. That is,
\[
{\C}_{2}^{n} \eqqcolon  \{(\xi_{1}, \xi_{2},\ldots,\xi_{n}) : \;\xi_{i} \in \C_2; \;  i = 1,2,\ldots,n\}.
\]
Furthermore, using idempotent representation, we can also express every element of $\C_2^n$ uniquely as:
\[
(\xi_{1}, \xi_{2},\ldots,\xi_{n}) =  (\xi_{1}^-, \xi_{2}^-,\ldots, \xi_{n}^-)e_{1}  +  (\xi_{1}^+,  \xi_{2}^+,\ldots,\xi_{n}^+)e_{2}
\]
such that $(\xi_{1}^-, \xi_{2}^-,\ldots, \xi_{n}^-)$ and $(\xi_{1}^+,  \xi_{2}^+,\ldots,\xi_{n}^+)$ are  n-tuples of complex numbers from the space $\C_1^n$.

\begin{remark} Some remarks can be easily made here.
\begin{enumerate}
\item Comparison between elements of $\C_2^{n}$ can be done by the following rules :
\begin{eqnarray*}
(\xi_{1}, \xi_{2},\ldots,\xi_{n})=(\eta_{1}, \eta_{2},\ldots, \eta_{n}) &\Longleftrightarrow & \xi_{i} = \eta_{i}  \quad \forall \; i\\
 &\Longleftrightarrow & \xi^{-}_{i} = \eta^{-}_{i}\;\& \;\; \xi^{+}_{i} = \eta^{+}_{i}\;\; \forall \; i\\
 i.e., \;\; (\xi_{1}, \xi_{2},\ldots,\xi_{n})=(\eta_{1}, \eta_{2},\ldots, \eta_{n}) &\Longleftrightarrow &  \xi^{-}_{i} = \eta^{-}_{i}\;\& \;\; \xi^{+}_{i} = \eta^{+}_{i}\;\; \forall \; i.
\end{eqnarray*}
\item 
 Generally, we know that if $F'$ is a sub-field of a field F then $F^n(F')$ forms a vector space. We have the result that $\C_2$ is not a field so the above result does not give us any proof that the $\C_2^{n}(\C_1)$ is a vector space. It is not hard to see that $\C_2^{n}$ forms a vector space over $\C_1$ with respect to usual addition and scalar multiplication. This immediately yields that the dimension of $\C_2^n$ over $C_1$ is $2n$, i.e., we have
\[
 \dim (\C_2^{n})=2n = 2 \cdot \dim(\C_1^n).
\]

\item Furthermore for an element $(\xi_{1},\xi_{2},\ldots,\xi_{n})\in \C_2^{n}$, we define a scalar product in $\C_2^n$  with elements $e_1,\;e_2$ by the rules: 
\begin{eqnarray*}
e_{1}\cdot (\xi_{1},\xi_{2},\ldots,\xi_{n}) = (\xi_{1},\xi_{2},\ldots,\xi_{n}) e_1  &\eqqcolon & (\xi_{1} e_{1}, \xi_{2} e_{1},\ldots, \xi_{n} e_{1})  = (\xi_{1}^{-} e_{1}, \xi_{2}^{-} e_{1},\ldots, \xi_{n}^{-} e_{1}), and\\
e_{2}\cdot (\xi_{1},\xi_{2},\ldots,\xi_{n})  = (\xi_{1},\xi_{2},\ldots,\xi_{n}) e_2 &\eqqcolon & (\xi_{1} e_{2}, \xi_{2} e_{2},\ldots, \xi_{n} e_{2}) = (\xi_{1}^{+} e_{2}, \xi_{2}^{+} e_{2},\ldots, \xi_{n}^{+} e_{2}).  
\end{eqnarray*}
It follows that naturally extends to a scalar product in $\C_2^n$ over $\C_2$ and hence to a product in $\C_2^n$ as follows:
\begin{eqnarray*}
\eta \cdot (\xi_{1},\xi_{2},\ldots,\xi_{n}) &=& (\eta \xi_{1},\eta \xi_{2},\ldots,\eta \xi_{n}), and\\
(\xi_{1},\xi_{2},\ldots,\xi_{n}) \cdot (\eta_{1},\eta_{2},\ldots,\eta_{n})&=&(\xi_1\eta_1,\cdots,\xi_n\eta_n).
\end{eqnarray*}
It makes $\C_2^n$ not just to be a $\C_2$-module but also to be a $\C_1$-algebra.
\end{enumerate}
\end{remark} 
 \noindent{\bf Cartesian idempotent product:}
A new kind of cartesian product in bicomplex spaces is hereby introduced which will be referred to as idempotent product which will be dealt and analysed in later part of this article. This helps us to switch from complex case to bicomplex case.\\
\indent Formally speaking, for any subsets \;$S_1,S_2 \subseteq \C_1^{n}$, their idempotent product, denoted as $S_1\times_e S_2$ is defined to be a subset of $\C_2^{n}$ given as
\begin{eqnarray}\label{idem.def}
S_1 \times_e S_2 \eqqcolon \{e_1 x + e_2 y : x \in S_1, y \in S_2\}. 
\end{eqnarray}
In the same manner, for any subsets \;$H_1,H_2 \subseteq \C_2^{m \times n}$, their idempotent product, denoted as $H_1\times_e H_2$ is defined to be a subset of $\C_2^{m \times n}$ given as
\begin{eqnarray}\label{idemmatrix.def}
H_1 \times_e H_2 \eqqcolon \{e_1 A + e_2 B : A \in H_1, B \in H_2\}. 
\end{eqnarray}

Similarly for any non-empty subsets $A_1,A_2$ of the space of \;$\C_1$-linear maps, $\Hom(\C_1^n,\C_1^m)$, their
idempotent product, written as $A_1\times_e A_2$, is defined to be a subset of $\Hom(\C_2^n,\C_2^m)$ given as
\begin{eqnarray}
A_1 \times_e A_2 \eqqcolon \{e_1 T_1 + e_2 T_2 : \;T_1 \in A_1, T_2 \in A_2\}. 
\end{eqnarray}
This induces a new formulation of bicomplex space $\C_2^n$ as\; $\C_2^n=\C_1^n\times_e \C_1^n$, for all $n\geqslant 1$. Also we get $\C_2^{m \times n}=\C_1^{m \times n}\times_e\C_1^{m \times n}$. We will later give a meaning to the elements of $A_1 \times_e A_2$.
\section{Bicomplex linear operator and bicomplex linear transformation}
In this section, we summarize bicomplex linear operator and we refer the reader to \cite{colombo2014bicomplex, kumar2015bicomplex} for further details. Also, we have defined bicomplex linear transformation. Let $V$ be a module over bicomplex space and Let $T \colon V  \to V $ be a map such that 
\begin{eqnarray*}
 T(u + v) = T(u) + T(v) \quad \mbox{and} \quad T(\alpha v) = \alpha T(v) \quad \forall \alpha \in \C_2  \quad \mbox{and} \quad u, v \in V. 
\end{eqnarray*}
Then we say that $T$ is a bicomplex linear operator on $V$. The set $End(V)$ of linear operator on $V$ forms a bicomplex - module by defining 
\begin{eqnarray*}
 (T +S)(v) &=& T(v) + S(v) \quad \forall \quad T,S \in End(V)    \\
(\alpha T)(v) &=& \alpha (T(v)) \quad \forall \quad T \in End(V), \alpha \in \C_2
\end{eqnarray*}
Let us set $V_1 = e_1 V = \{e_1 v : v \in v\}$. Any element $v \in V$ can be written as
\begin{eqnarray*}
    v = (e_1 + e_2)v = e_1 v + e_2 v = v_1 + v_2 = e_1 v_1 + e_2 v_2, \quad \mbox{where} \quad v_1 = e_1 v,  v_2 =e_2 v
\end {eqnarray*}
It is evident that $V = V_1 \oplus V_2$. We can define the operators $T_1 , T_2$ as 
\begin{eqnarray*}
  T_1(v) = e_1 T(v), T_2(v) = e_2 T(v)  
\end{eqnarray*}
where $T_1 \colon V  \to V_1 $ and $T_2 \colon V  \to V_2 $
\begin{remark}
    The following properties hold
    \begin{enumerate}
\item The operators $T_1$ and $T_2$ are bicomlex linear
\item we get the decomposition $T = e_1 T_1 + e_2 T_2$
\item The action $T$ on $V$ can be decomposed as fellows\\
$T(v) = e_1 T_1(v_1) + e_2 T_2(v_2)$, where $v = e_1 v_1 + e_2 v_2 \in V$
\item The operator $T_k \colon V_k  \to V_k ; k = 1,2$
\end{enumerate}
\end{remark}
\begin{definition}\label{definition2.2}(Bicomplex linear transformation)
If $V$ and $V'$ is a module over $\C_2$ and let $T \colon V  \to V'$ be a map such that 
\begin{eqnarray*}
    T(u + v) &=& T(u) + T(v)\\
    T(\alpha u) &=& \alpha T(u) \quad \forall \quad \alpha \in \C_2 \quad \mbox{and} \quad u,v \in V
\end{eqnarray*}
then we say that $T$ is a bicomplex linear transformation from V to V'    
\end{definition}
\indent As previously, we can define the sets $V_1, V_2, V'_1, V'_2$ and the linear transformation $T_1$ and $T_2$ here as well
\[
V_1 = e_1 V ,\quad V_2 = e_2 V , \quad  V'_1 = e_1 V' \quad \mbox{and} \quad V'_2 = e_2 V'
\]
\indent Any element $v \in V$ and $v' \in V'$ can be written as 
\[
v = e_1 v_1 + e_2 v_2 , \quad v' = e_1 v'_1 + e_2 v'_2 \quad \mbox{where} \quad  v_i = e_i v , v'_i = e_i v' \forall i = 1,2
\]
\indent Also it is evident that $V = V_1 \oplus V_2$ and $V' = V'_1 \oplus V'_2$. The linear transformation  $T_1 \colon V  \to   V'_1$ and $T_2 \colon V  \to   V'_2$ will be 
\[
T_1(v) = e_1 T(v) \quad \mbox{and} \quad T_2(v) = e_2 T(v)
\]
\begin{remark}
    We can easily prove the following properties
\begin{enumerate}
\item The linear transformation $T_1$ and $T_2$ are bicomplex linear.
\item We get the decomposition $T = e_1 T_1 + e_2 T_2$
\item The action $T$ on $V$ can be decomposed as fellows
\[
T(v) = e_1 T_1(v_1) + e_2 T_2(v_2); \quad v = e_1 v_1 + e_2 v_2 \in V
\]
\item The linear transformation $T_k \colon V_k  \to   V'_k; k =1,2 $
\end{enumerate}
\end{remark}
\indent particularly if we set $V = \C_2^n$ and $V' = \C_2^m$ then $V_1 = e_1 \C_1^n, V_2 = e_2 \C_1^n$ and 
$V'_1 = e_1 \C_1^m, V'_2 = e_2 \C_1^m$. In this case $T \colon \C_2^n  \to  \C_2^m $ is a bicomplex linear transformation. Also, $T_1 \colon \C_2^n  \to  e_1 \C_1^m$ and $T_2 \colon \C_2^n  \to  e_2 \C_1^m$ specially, $T_1 \colon e_1 \C_1^n  \to  e_1 \C_1^m$  and $T_2 \colon e_2 \C_1^n  \to  e_2 \C_1^m$ are bicomplex linear transformation.

\begin{remark}
    If $(\xi_{1}, \xi_{2},\ldots,\xi_{n}) \in \C_2^n$ and 
\begin{eqnarray}
     T(\xi_{1}, \xi_{2},\ldots,\xi_{n}) =(\eta_{1}, \eta_{2},\ldots,\eta_{m})   
\end{eqnarray}
As $T$ is linear with respect to bicomplex numbers then 
\begin{eqnarray*}
e_1 T(\xi_{1}^-, \xi_{2}^-,\ldots,\xi_{n}^-) +  e_2 T(\xi_{1}^+, \xi_{2}^+,\ldots,\xi_{n}^+) = e_1 (\eta_{1}^-, \eta_{2}^-,\ldots,\eta_{m}^-) + e_2 (\eta_{1}^+, \eta_{2}^+,\ldots,\eta_{m}^+)
\end{eqnarray*}
\begin{eqnarray}
 \Rightarrow T(\xi_{1}^-, \xi_{2}^-,\ldots,\xi_{n}^-) = (\eta_{1}^-, \eta_{2}^-,\ldots,\eta_{m}^-)   
\end{eqnarray}

That is
\begin{eqnarray}
e_1 T(\xi_{1}^-, \xi_{2}^-,\ldots,\xi_{n}^-) &=& e_1 (\eta_{1}^-, \eta_{2}^-,\ldots,\eta_{m}^-)\\
\mbox{and} \quad T[e_1 (\xi_{1}^-, \xi_{2}^-,\ldots,\xi_{n}^-)] &=& e_1 (\eta_{1}^-, \eta_{2}^-,\ldots,\eta_{m}^-)
\end{eqnarray}
Also, by using definition \ref{definition2.2}, $(6)$ and $(7)$, we have
\begin{eqnarray}
T_1[e_1 (\xi_{1}^-, \xi_{2}^-,\ldots,\xi_{n}^-)] &=& e_1 T[e_1 (\xi_{1}^-, \xi_{2}^-,\ldots,\xi_{n}^-)] \nonumber \\
&=& e_1 e_1 (\eta_{1}^-, \eta_{2}^-,\ldots,\eta_{m}^-) \nonumber \\
\therefore \quad T_1[e_1 (\xi_{1}^-, \xi_{2}^-,\ldots,\xi_{n}^-)] &=& e_1 (\eta_{1}^-, \eta_{2}^-,\ldots,\eta_{m}^-) \nonumber \\
&=&e_1 T(\xi_{1}^-, \xi_{2}^-,\ldots,\xi_{n}^-)\nonumber\\
&=& T_1(\xi_{1}^-, \xi_{2}^-,\ldots,\xi_{n}^-) \nonumber \\
&=& e_1 (\eta_{1}^-, \eta_{2}^-,\ldots,\eta_{m}^-) 
\end{eqnarray}
Now
\begin{eqnarray}
T_1(\xi_{1}, \xi_{2},\ldots,\xi_{n}) &=& e_1 T (\xi_{1}, \xi_{2},\ldots,\xi_{n}) \nonumber\\
&=& e_1 (\eta_{1}, \eta_{2},\ldots,\eta_{m}) \nonumber\\
\therefore \quad T_1(\xi_{1}, \xi_{2},\ldots,\xi_{n}) &=& e_1 (\eta_{1}^-, \eta_{2}^-,\ldots,\eta_{m}^-)
\end{eqnarray}
$\therefore$ from $(6), (7), (8) $ and (9)
\begin{eqnarray*}
    T_1(\xi_{1}, \xi_{2},\ldots,\xi_{n}) &=& e_1 (\eta_{1}^-, \eta_{2}^-,\ldots,\eta_{m}^-) \\
    &=& e_1 T(\xi_{1}^-, \xi_{2}^-,\ldots,\xi_{n}^-)\\
    &=& T [e_1 (\xi_{1}^-, \xi_{2}^-,\ldots,\xi_{n}^-)]\\
    &=& T_1 [e_1 (\xi_{1}^-, \xi_{2}^-,\ldots,\xi_{n}^-)]\\
    &=& T_1 (\xi_{1}^-, \xi_{2}^-,\ldots,\xi_{n}^-)
\end{eqnarray*}
Similarly 
\begin{eqnarray*}
    T_2(\xi_{1}, \xi_{2},\ldots,\xi_{n}) &=& e_2(\eta_{1}^+, \eta_{2}^+,\ldots,\eta_{m}^+) \\
    &=& e_2 T(\xi_{1}^+, \xi_{2}^+,\ldots,\xi_{n}^+)\\
    &=& T [e_2 (\xi_{1}^+, \xi_{2}^+,\ldots,\xi_{n}^+)]\\
    &=& T_2 [e_2 (\xi_{1}^+, \xi_{2}^+,\ldots,\xi_{n}^+)]\\
    &=& T_2 (\xi_{1}^+, \xi_{2}^+,\ldots,\xi_{n}^+)
\end{eqnarray*}
\end{remark}
\indent It is worth noting here that if $T \colon \C_2^n  \to  \C_2^m$ is a bicomplex linear transformation then we get the bicomplex linear transformation $T_1 \colon e_1 \C_1^n  \to  e_1 \C_1^m$ and $T_2 \colon e_2 \C_1^n  \to  e_2 \C_1^m$ and vice versa. This notation is possible because $T$ is linear with respect to bicomplex numbers. Furthermore, we have defined another map in definition $3.2$ for which this notion is not true because the map define in there is not linear with respect to bicomplex numbers.

\section{Bicomplex Matrices and Linear transformations}
This section focuses on bicomplex matrices and linear maps defined on the space $\C_2^n$. It connects both these notions like the ordinary ones on complex numbers.

\indent A bicomplex matrix of order $m\times n$ is denoted by \;$[\xi_{i j}]_{m \times n}$, $\xi_{i j} \in \C_2$\; or simply by $[\xi_{i j}]$ if there is no confusion.  We let the set of all bicomplex matrices of order $m \times n$ be denoted by $\C_2^{m \times n}$, i.e., we have
\begin{eqnarray}
\C_2^{m \times n}  \eqqcolon  \Big\{[\xi_{i j}] : \; \xi_{i j} \in \C_2 ;\; i = 1,2,\ldots,m,\; j = 1,2,\ldots,n \Big\}.
\end{eqnarray} 
The usual matrix addition and scalar multiplication makes the space  $\C_2^{m \times n}$ to be a vector space over the field $\C_1$. Further each bicomplex matrix $[\xi_{i j}]$ can be uniquely written as
\begin{eqnarray}
 [\xi_{i j}]=e_1 \ [\xi^-_{i j}] \ + \ e_2 \  [\xi^+_{i j}],
\end{eqnarray}
where  $[\xi^-_{i j}],\;[\xi^+_{i j}]$ are complex matrices of order $m\times n$. From this, it follows easily that
\begin{eqnarray} \label{2mn}
\dim(\C_2^{m \times n})=  2 \cdot \dim (\C_1^{m \times n}) = 2mn.
\end{eqnarray}

\noindent We now come to linear maps, i.e., linear maps on bicomplex spaces. As we know that $\Hom(\C_1^{n},\C_1^{m})$ represents the set of all $\C_1$-linear maps from $\C_1^{n}$ to $\C_1^{m}$ so does $\Hom(\C_2^{n},\C_2^{m})$. 
However we simplify these notations to our comfort.

\begin{remark} \label{rem1}
For brevity and convenience, we denote the set of all $\C_1$-linear maps from $\C_1^{n}$ to $\C_1^{m}$ by $L_1^{nm}$, instead of $\Hom(\C_1^{n},\C_1^{m})$ and set  of all $\C_1$-linear maps from $\C_2^{n}$ to $\C_2^{m}$ by $L_2^{nm}$.  Clearly both $L_1^{nm}$ and  $L_2^{nm}$ are vector spaces over $\C_1$. Hence, we can see that 
\begin{eqnarray}\label{4mn}
\dim(L_1^{nm}) = m n \;\;\mbox{and}\;\; \dim(L_2^{nm}) = \dim C_2^{n}\cdot \dim C_2^{m} = 2n\cdot 2m = 4mn. 
\end{eqnarray}

It is evident that $L_1^{nm}$=$\Hom(\C_1^{n},\C_1^{m})$ is isomorphic to $\C_1^{m \times n}$ as $\C_1$ is a field but here we have $\C_2$ is not a field so we cannot say that $L_2^{nm}$=$\Hom(\C_2^{n},\C_2^{m})$ and $\C_2^{m \times n}$ are isomorphic. (\ref{2mn}) and (\ref{4mn}) follow that $L_2^{nm}$ is not isomorphic to $\C_2^{m \times n}$. Moreover $\C_2^{m \times n}$ is a proper subspace of $L_2^{nm}$ in isomorphic sense. So we try to find a subset of $L_2^{nm}$ which is isomorphic to $\C_2^{m \times n}$. Thus the following definition.
\end{remark}

\begin{definition}\label{product}
For any given $T_{1}, T_{2} \in L_1^{nm}$, we define a map $f \colon \C_2^{n}  \to \C_2^{m} $  by the following rule:
\[
f(\xi_{1}, \xi_{2},\ldots,\xi_{n}) \eqqcolon e_{1}\cdot {T}_{1} (\xi_{1}^-, \xi_{2}^-,\ldots,\xi_{n}^-)
+ e_{2}\cdot {T}_{2} (\xi_{1}^+, \xi_{2}^+,\ldots,\xi_{n}^+).
\]
One can notice that \;$f$\; is a $\C_1$-linear map. This \;$f$\; is defined as $e_{1} T_{1} + e_{2} T_{2}$. Therefore, the set of all such linear maps is the idempotent product $L_1^{nm} \times_{e} L_1^{nm}$, i.e., we have
\begin{eqnarray}\label{idemproduct}
L_1^{nm} \times_{e} L_1^{nm} &\eqqcolon & \left \{\;e_{1} T_{1} + e_{2} T_{2} \in L_2^{nm}:\; \;T_{1}, T_{2} \in L_1^{nm}\right \}.  
\end{eqnarray}
\end{definition}
\begin{figure}[h]
	\centering
	\begin{tikzpicture}[node distance=2.5cm]
		\node (alc)[process,xshift=2cm,draw=black, fill=white!20] {$(\xi_1^-,\xi_2^-, \cdots,\xi_n^-)\in \C_1^n$};
		\node (bob)[process, right of=alc, xshift=5cm,draw=black, fill=white!20] {$T_1(\xi_1^-,\xi_2^-,\cdots,\xi_n^-)\in \C_1^m$};
		\node (iob1) [process,draw=black, below of=bob,xshift=-7.5cm, yshift=-0.5cm] {$(\xi_1,\xi_2, \cdots,\xi_n)\in \C_2^n$}; 
   \node (iob2) [process, draw=black, below of=bob,xshift=.1cm, yshift=-0.5cm]{$e_1 T_1(\xi_1^-,\xi_2^-,\cdots,\xi_n^-)$ \\ 
    + 
   $e_2 T_2(\xi_1^+,\xi_2^+,\cdots,\xi_n^+)$
        \\ $\in L_1 \times_e L_1$};	
		\node (b1)[process, below of=alc, yshift=-3.5cm, draw=black, fill=white!20] {$(\xi_1^+,\xi_2^+, \cdots,\xi_n^+)\in \C_1^n$};
		\node (b2)[process, below of=bob, yshift=-3.5cm, draw=black, fill=white!20] {$T_2(\xi_1^+,\xi_2^+, \cdots,\xi_n^+)\in \C_1^m$};
			\draw[->] (iob1)-- (alc);
			\draw[->] (alc)-- node[anchor=south] {$T_1$} (bob);
			\draw[->] (bob)-- (iob2);
			\draw[->] (iob1)--node[anchor=south] {$e_1T_1+e_2T_2$} (iob2);
				\draw[->] (iob1)-- (b1);
			\draw[->] (b1)-- node[anchor=south] {$T_2$} (b2);
			\draw[->] (b2)-- (iob2);
	\end{tikzpicture}
\caption{Diagram showing the idempotent components of  $T_{1}$ and $T_{2}$ which are compatible with the idempotent components $(\xi_{1}^-, \xi_{2}^-,\ldots,\xi_{n}^-)$ and $(\xi_{1}^+, \xi_{2}^+,\ldots,\xi_{n}^+)$}
\end{figure}
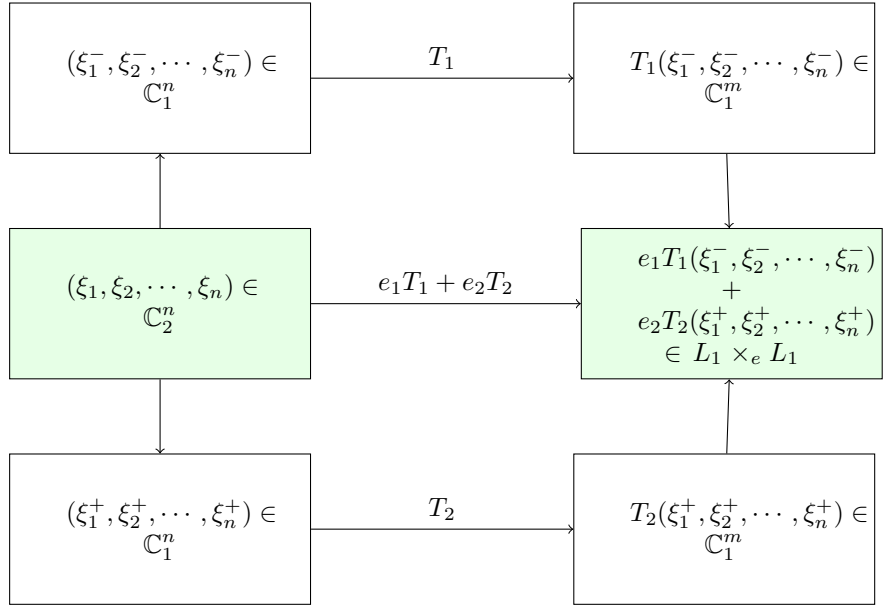

\noindent It is a routine matter to check that the idempotent product $L_1^{nm} \times_{e} L_1^{nm}$ is a subspace of $L_2^{nm}$. This follows easily from the following proposition.

\begin{proposition} \label{Proposition1}
Let $T, S \in L_{1}^{nm} \times_{e} L_{1}^{nm}$ be any elements such that $T = e_{1} T_{1} + e_{2} T_{2}$ and $S = e_{1} S_{1} + e_{2} S_{2}$. Then, we have 
\begin{enumerate}
\item $T + S = e_{1} (T_{1} + S_{1}) + e_{2} (T_{2} + S_{2})$
\item $\alpha T = e_{1} (\alpha T_{1}) + e_{2} (\alpha T_{2}); \quad \forall \alpha \in \C_1$
\end{enumerate}
\end{proposition}

\noindent The following theorem contains some basic properties of the elements of  $L_{1}^{nm} \times_{e} L_{1}^{nm}$.
\begin{theorem}\label{properties}{\bf (Properties)} 
Let $T = e_{1} T_{1} + e_{2} T_{2},\; S = e_{1} S_{1} + e_{2} S_{2}$ be any two elements of $L_{1}^{nm} \times_{e} L_{1}^{nm}$. Then, we have
\begin{enumerate}
\item $T = 0$ \IFF \; $T_{1} = 0, T_{2} = 0$
\item $T = S$ \IFF  \; $T_{1} = S_{1}, T_{2} = S_{2}$
\item $S\circ T = e_{1} (S_{1} \circ T_{1}) + e_{2} (S_{2} \circ T_{2})$, wherever composition defined.
 \end{enumerate}
 \end{theorem}
\noindent The next theorem deals with the dimension of $L_{1}^{nm} \times_{e} L_{1}^{nm}$.
\begin{theorem}\label{dim1}
The dimension of $L_{1}^{nm} \times_{e} L_{1}^{nm}$ is equal to \;$2\cdot \dim L_{1}^{nm}$
\end{theorem}

\begin{proof}
As we know that $\dim L_{1}^{nm} = m n$. So we take $\mathcal B_{1} = \{T_{j}:\;j=1,2,\cdots,mn\}$ as a basis for $L_{1}^{nm}$. This yields a collection $\mathcal B_{2} = \{ e_{i} T_{j}:i=1,2\;\&\;j=1,2,\cdots,mn\}$ of elements of $L_{1}^{nm} \times_{e} L_{1}^{nm}$. We assert that $\mathcal B_{2}$ is a basis for $L_{1}^{nm} \times_{e} L_{1}^{nm}$. For any $S\in L_{1}^{nm} \times_{e} L_{1}^{nm}$ there are $S_{1}, S_{2} \in L_{1}^{nm}$ such that $S=e_1S_1+e_2S_2$. Since \;$<{\mathcal B_{1}}> = L_{1}^{nm}$, there exist $\alpha_j, \beta_{j} \in \C_1 ;\; j = 1,2,\ldots,mn$ such that 
\begin{eqnarray*}
S_{1} = \sum_{j=1}^{mn} \alpha_{j} T_{j} &\mbox{and}& S_{2} = \sum_{j=1}^{mn} \beta_{j} T_{j}\\
\Rightarrow\quad S = \sum_{j=1}^{mn} e_{1}(\alpha_{j} T_{j}) + \sum_{j=1}^{mn} e_{2}(\beta_{j} T_{j}) &=& \sum_{j=1}^{mn} \alpha_{j} (e_{1} T_{j}) + \sum_{j=1}^{mn} \beta_{j} (e_{2} T_{j})\\
\Rightarrow \quad S \in <\mathcal B_{2}> \;  \;&\Rightarrow &\;<\mathcal B_{2}>\; = L_{2} \times_{e} L_{2}.
\end{eqnarray*}
Furthermore, if we consider
\[
\sum_{j=1}^{mn} \alpha_{j} (e_{1} T_{j} ) + \sum_{j=1}^{mn} \beta_{1} (e_{2} T_{i}) = 0  \quad \Rightarrow  \quad e_{1} \sum_{j=1}^{mn} (\alpha_{j} T_{j}) + e_{2} \sum_{j=1}^{mn} (\beta_{j} T_{j}) = 0.
\]
Using property (1) of  theorem \ref{properties}, we get
\[
\sum_{j=1}^{mn} (\alpha_{j} T_{j}) = 0 \quad \& \quad \sum_{j=1}^{mn} (\beta_{1} T_{j}) = 0 \; \Rightarrow \; \alpha_{j} = 0=\beta_{j} \;\;\forall\; j.
\]
Hence, $\mathcal B_{2}$ is linearly independent set. Hence the result follows.
\end{proof}

\begin{remark} \label{remark}
(\ref{4mn}) and theorem \ref{dim1} show that $L_1^{nm} \times_{e} L_1^{nm}$ is a proper subspace of $L_2^{nm}$.
It is clear that $\C_2^{m \times n} = \C_1^{m \times n} \times_e \C_1^{m \times n}$ and $\C_2^n = \C_1^n \times_e \C_1^n$. But $\Hom(\C_2^n , \C_2^m) \neq \Hom(\C_1^n ,\C_1^m) \times_e \Hom(\C_1^n ,\C_1^m)$. In fact, $\Hom(\C_1^n ,\C_1^m) \times_e \Hom(\C_1^n ,\C_1^m) \subset \Hom(\C_2^n , \C_2^m)$. This will help us to defining one - one correspondence between bicomplex matrix and linear map.
\end{remark}

\section{Relationship between bicomplex linear map and matrix}
\noindent The studying of $\C_1$-linear maps is equivalent to that of complex matrices. Is the same true for a bicomplex matrix? In this section we work on it in the context of bicomplex matrix. Let us suppose that $V$ and $V'$ be two vector space of dimension $n$ and $m$ respectively, over the field $F$. Also we have $\mathcal{B}$ and $\mathcal{B'}$ are two ordered bases of vector space $V$ and $V'$ respectively. Then the vector spaces ${\Hom ({V,V'})}(F)$ and $F^{m \times n}(F)$ are isomorphic two each other. If $T$ is any linear map in $\Hom ({V,V'})$ then the matrix representation of $T$ with respect two ordered bases $\mathcal{B}$ and $\mathcal{B'}$ corresponds to a unique matrix in $F^{m \times n}(F)$ and vice-versa. 
This unique correspondence help us to solving system of equation and  to find rank, nullity , eigen value ,eigen vector , characteristic polynomial and minimal polynomial. 

If we assume that the field is a set of complex number then for any complex matrix of order $m \times n$ there will be a unique linear map from vector spaces $V$ to $V'$. The vector space $V$ and $V'$ can be substitute by vector space $\C_1^n $ and $\C_1^m$ because of $V$ and $\C_1^n $, $V'$ and $\C_1^m $ have same dimension. So whenever we want to study on complex matrix of order $m \times n $ then we can study its corresponding linear map $T \colon \C_1^n  \to \C_1^m$ to get the information about concern matrix. This analogous  concept can not be define for a bicomplex matrix of order $m \times n $ and the linear map $T \colon \C_2^n  \to \C_2^m$. As we know $\C_2$ is not a field so a matrix representation of any linear map can not give a bicomplex matrix and \ref{rem1} shows that $L_2^{nm}$ is not isomorphic to $\C_2^{m \times n}$. In the same manner the matrix representation of the linear map $T \colon \C_2^n  \to \C_2^m$ will be a complex matrix of order $2m \times 2n$. So using the traditional approach of a matrix representation of a linear map we can not develop the analogous concept for a bicomplex matrix and the linear map $T$.

To do away this problem we define a new approach called "Idempotent method" for matrix representation of a linear map.

\begin{definition}\label{methodl}(Idempotent method)
Let $\mathcal{B}_{1}, \mathcal B_{2}$ be the ordered bases for $\C_1^{n}$ and $\C_1^{m}$ respectively. Then for any linear map $T \in L_{1}^{nm} \times_{e} L_{1}^{nm}$ so that $T = e_{1} T_{1} + e_{2} T_{2}$ for some $T_1,T_2\in L_1^{nm}$, the matrix representation of $T$ with respect to these bases is defined as 
\begin{eqnarray}
[T]^{\mathcal{B}_{1}}_{\mathcal{B}_{2}}\eqqcolon e_1 [T_1]^{\mathcal{B}_{1}}_{\mathcal{B}_{2}} + e_{2} [T_2]^{\mathcal{B}_{1}}_{\mathcal{B}_{2}}.
\end{eqnarray}
This new approach will help us to find the solution of a system of equations of bicomplex variables and we can define the rank, nullity, eigenvalue, eigenvector, characteristic polynomial, and minimal polynomial for a bicomplex matrix. Here we have not used $T$ as defined in definition $2.2$. The reason behind this is that the $T_1$ and $T_2$ corresponding to $T$ defined in definition $2.2$ are bicomplex linear transformations from  $e_1 \C_1^n  \to  e_1 \C_1^m$ and  $e_2 \C_1^n  \to  e_2 \C_1^m$ respectively. So, the definition $2.2$ is not more useful for our study.
\end{definition}

\begin{remark}
In the special case when $\C_1^{n} = \C_1^{m}$, the matrix representation of the operator $T= e_{1} T_{1} + e_{2} T_{2}$ with respect to basis $\mathcal{B}$ for $\C_1^{n}$ is simplified to  $[T]_{\mathcal{B}}$ from $[T]^{\mathcal{B}}_{\mathcal{B}}$. Thus, we have
\begin{eqnarray}\label{matrixrepre}
[T]_{\mathcal{B}} = e_{1} [T_{1}]_{\mathcal{B}} + e_{2} [T_{2}]_{\mathcal{B}}.
\end{eqnarray}
\end{remark}

\begin{example}
Let $T_1, T_2  \colon \C_1^{2} \to \C_1^{3}$ be linear maps defined as 
\[
T_1(a,b)=(a, a +b, b)\quad\mbox{and}\quad T_2(a,b) = (a-b, b, a)\quad\forall\;(a,b)\in \C_1^2.
\]
Let $\mathcal{B}_{1} = \{(1,1),(1,0)\}$ and $\mathcal{B}_{2}=\{(1,0,1),(1,1,0),(0,0,1)\}$ be the ordered bases for  $\C_1^{2}$ and $\C_1^{3}$ respectively. This gives 
\begin{eqnarray*}
\left[T_{1}\right]^{\mathcal{B}_{1}}_{\mathcal{B}_{2}} = \begin{bmatrix}
-1 & 0 \\ 2 & 1 \\2 & 0\end{bmatrix} &\mbox{and}& \left[T_{2}\right]^{\mathcal{B}_{1}}_{\mathcal{B}_{2}}= \begin{bmatrix} -1 & 1 \\ 1 & 0 \\ 2 & 0 \end{bmatrix}\\
\Rightarrow \quad e_{1} [T_{1}]^{\mathcal{B}_{1}}_{\mathcal{B}_{2}} + e_{2} [T_{2}]^{\mathcal{B}_{1}}_{\mathcal{B}_{2}} &=& e_1\begin{bmatrix}
-1 & 0 \\ 2 & 1 \\2 & 0 \end{bmatrix} + e_2 \begin{bmatrix}
-1 & 1 \\ 1 & 0 \\ 2 & 0 \end{bmatrix} =\begin{bmatrix}
-1 & e_2 \\ 2 e_1 +  e_2 & e_1 \\ 2 & 0 \end{bmatrix}.
\end{eqnarray*}
Therefore, the matrix representation of linear transformation $T = e_1 T_1 + e_2 T_2$ with respect to $\mathcal{B}_{1},\;\mathcal{B}_{2}$ is given by
\begin{center}
$[T]^{\mathcal{B}_{1}}_{\mathcal{B}_{2}} = 
\begin{bmatrix}
-1 & e_2 \\ 2 e_1 +  e_2 & e_1 \\ 2 & 0 \end{bmatrix}$.
\end{center}
\end{example}

Next we show that the matrix representation as defined in (\ref{methodl}) behaves well with respect to both the operations in $L_{1}^{nm} \times_{e} L_{1}^{nm}$.
\begin{theorem}
Let $T = e_{1} T_{1} + e_{2} T_{2}, \ S = e_{1} S_{1} + e_{2} S_{2}$ be any two elements of $L_{1}^{nm} \times_{e} L_{1}^{nm}$ and let $\mathcal{B}_{1}, \mathcal{B}_{2}$ be the ordered bases of\; $\C_1^{n}$ and $\C_1^{m}$ respectively. Then 
\begin{enumerate}
\item 
$[T + S]^{\mathcal{B}_{1}}_{\mathcal{B}_{2}} = [T]^{\mathcal{B}_{1}}_{\mathcal{B}_{2}} + [S]^{\mathcal{B}_{1}}_{\mathcal{B}_{2}}$
\item
$[\alpha T]^{\mathcal{B}_{1}}_{\mathcal{B}_{2}} = \alpha [T]^{\mathcal{B}_{1}}_{\mathcal{B}_{2}} \forall \ \alpha \in \C_1$
\end{enumerate}
\end{theorem}
\begin{proof}
By definition \ref{methodl} and using proposition \ref{Proposition1}, it follows that
\begin{eqnarray*}
[T + S]^{\mathcal{B}_{1}}_{\mathcal{B}_{2}} &=& e_{1} [T_{1} + S_{1}]^{\mathcal{B}_{1}}_{\mathcal{B}_{2}} + e_{2} [T_{2} + S_{2}]^{\mathcal{B}_{1}}_{\mathcal{B}_{2}}\\
&=& e_{1} 
{\left([T_{1}]^{\mathcal{B}_{1}}_{\mathcal{B}_{2}}+[S_{1}]^{\mathcal{B}_{1}}_{\mathcal{B}_{2}} \right) + e_{2}\left( [T_{2}]^{\mathcal{B}_{1}}_{\mathcal{B}_{2}} +  [S_{2}]^{\mathcal{B}_{1}}_{\mathcal{B}_{2}} \right)} \\
&=& \left(e_{1} [T_{1}]^{\mathcal{B}_{1}}_{\mathcal{B}_{2}} + e_{2} [T_{2}]^{\mathcal{B}_{1}}_{\mathcal{B}_{2}}\right) + \left(e_{1} [S_{1}]^{\mathcal{B}_{1}}_{\mathcal{B}_{2}} + e_{2} [S_{2}]^{\mathcal{B}_{1}}_{\mathcal{B}_{2}}\right) \\  
&=& [T]^{\mathcal{B}_{1}}_{\mathcal{B}_{2}} + [S]^{\mathcal{B}_{1}}_{\mathcal{B}_{2}}.
\end{eqnarray*}
For the second part, proposition \ref{Proposition1} gives that
\begin{eqnarray*}
[\alpha T]^{\mathcal{B}_{1}}_{\mathcal{B}_{2}} &=& e_{1} [\alpha T_{1}]^{\mathcal{B}_{1}}_{\mathcal{B}_{2}} + e_{2} [\alpha T_{2}]^{\mathcal{B}_{1}}_{\mathcal{B}_{2}} \\
 &=& e_{1} \left(\alpha [T_{1}]^{\mathcal{B}_{1}}_{\mathcal{B}_{2}}\right) + e_{2} \left(\alpha [T_{2}]^{\mathcal{B}_{1}}_{\mathcal{B}_{2}}\right)\\
 &=&  \alpha \left(e_{1} [T_{1}]^{\mathcal{B}_{1}}_{\mathcal{B}_{2}} + e_{2} [T_{2}]^{\mathcal{B}_{1}}_{\mathcal{B}_{2}}\right)  \\ 
&=& \alpha [T]^{\mathcal{B}_{1}}_{\mathcal{B}_{2}}. \hspace{3in} 
\end{eqnarray*}
This completes the proof.
\end{proof}
\begin{theorem}
The space\; $L_{1}^{nm} \times_{e} L_{1}^{nm}$ and \;$\C_2^{m \times n}$ are isomorphic.
\end{theorem}
\begin{proof}

let $T = e_{1} T_{1} + e_{2} T_{2} \in L_1^{nm} \times_e L_1^{nm}$. Define a function $F \colon L_1^{nm} \times_e L_1^{nm} \to \C_2^{m \times n}$ such that 
\[ F(T) = [T]^{\mathcal{B}_{1}}_{\mathcal{B}_{2}}.
\]
It is an easy exercise to show that $F$ is one -one onto linear map.
Moreover in view of (\ref{2mn}) and theorem \ref{dim1}, the vector spaces $\C_2^{m \times n}$ and $L_{1}^{nm} \times_{e} L_{1}^{nm}$ have the same dimension over the field $\C_1$.  So they are isomorphic.
\end{proof}

\noindent As per the definition of idempotent product given in (\ref{idem.def}), we have
\begin{eqnarray}
 \ker T_{1} \times_e \ker T_{2} &\eqqcolon & \{ e_{1} z + e_{2} w\in \C_2^n\;:\; z \in \ker T_{1}, \;w \in \ker T_{2}\},\\
 \ran T_{1} \times_{e} \ran T_{2} &\eqqcolon & \{ e_{1} z + e_{2} w\in \C_2^m\;:\; z \in \ran T_{1}\;w \in \ran T_{2}\}.
\end{eqnarray}
The next theorem connect both kernel and range of of the element of $L_{1}^{nm} \times_{e} L_{1}^{nm}$ with the kernels  and ranges of their components.
\begin{theorem}\label{Kernel}
Let $T \in L_{1}^{nm} \times_{e} L_{1}^{nm}$ so that $T=e_{1} T_{1} + e_{2} T_{2}$ for some $T_1,T_2\in L_1^{nm}$. Then, we have
\begin{enumerate}
\item $\ker(e_{1} T_{1} + e_{2}  T_{2}) = \ker T_{1} \times_{e} \ker T_{2} $
\item $\ran (e_{1} T_{1} + e_{2}  T_{2}) = \ran T_{1} \times_{e} \ran T_{2}  $
\end{enumerate}
\end{theorem}
\begin{proof} Let $(\xi_{1},\xi_{2},\ldots,\xi_{n}) \in \ker(e_{1} T_{1} + e_{2} T_{2})$. Then $\left(e_{1} T_{1} + e_{2} T_{2}\right)(\xi_{1}, \xi_{2},\ldots,\xi_{n}) = 0$. Using definition \ref{product} and from the part (1) of proposition \ref{properties}, it will follows that 
\begin{eqnarray*}
\left(e_{1} T_{1} + e_{2} T_{2}\right)(\xi_{1}, \xi_{2},\ldots,\xi_{n}) &=& e_1\cdot {T}_{1} (\xi_{1}^-, \xi_{2}^-,\ldots,\xi_{n}^-) + e_{2} \cdot {T}_{2} (\xi_{1}^+, \xi_{2}^+,\ldots,\xi_{n}^+) = 0\\
\Leftrightarrow \quad {T}_{1} (\xi_{1}^-, \xi_{2}^-,\ldots,\xi_{n}^-) = 0 &{\&}& {T}_{2} (\xi_{1}^+, \xi_{2}^+,\ldots,\xi_{n}^+) = 0\\
\Leftrightarrow \quad (\xi_{1}^-, \xi_{2}^-,\ldots,\xi_{n}^-) \in \ker {T}_{1} &{\&}& (\xi_{1}^+, \xi_{2}^+,\ldots,\xi_{n}^+)\in \ker {T}_{2}\\
\Leftrightarrow \quad (\xi_{1},\xi_{2},\ldots,\xi_{n}) \in && \ker T_{1} \times_{e} \ker T_{2}.
\end{eqnarray*}
For the secondly part, consider $(\xi_{1},\xi_{2},\ldots,\xi_{m}) \in \ran T_{1} \times_{e} \ran T_{2}$, then we have
\begin{eqnarray*}
(\xi_{1}^-, \xi_{2}^-,\ldots,\xi_{m}^-) \in \ran T_{1} &and& (\xi_{1}^+, \xi_{2}^+,\ldots,\xi_{m}^+) \in \ran T_{2}\\
\Rightarrow  \quad  \exists \quad  (z_1,z_2,\ldots,z_n)&,& (w_1,w_2,\ldots,w_n) \in \C_1^n  \\
\mbox{such that}\quad 
T_1(z_1,z_2,\ldots,z_n) = (\xi_{1}^-, \xi_{2}^-,\ldots,\xi_{m}^-) &and& T_2(w_1,w_2,\ldots,w_n)= (\xi_{1}^+, \xi_{2}^+,\ldots,\xi_{m}^+)\\
\mbox{So}, \quad   
e_1 (\xi_{1}^-, \xi_{2}^-,\ldots,\xi_{m}^-) + e_2 (\xi_{1}^+, \xi_{2}^+,\ldots,\xi_{m}^+) &\in&  \ran (e_{1} T_{1} + e_{2}  T_{2})\\
\mbox{Hence}  \quad \ran T_1 \times_e \ran T_2 &\subseteq& \ran (e_{1} T_{1} + e_{2}  T_{2}).
\end{eqnarray*}
 Conversely  suppose $e_1 T_1(\xi_{1}^-, \xi_{2}^-,\ldots,\xi_{n}^-) + e_2 T_2(\xi_{1}^+, \xi_{2}^+,\ldots,\xi_{n}^+) \in  \ran (e_{1} T_{1} + e_{2}  T_{2})$, for some $(\xi_{1},\xi_{2},\ldots,\xi_{n})  \in \C_2^n $. This follows that $ \ran (e_{1} T_{1} + e_{2}  T_{2}) \subseteq \ran T_1 \times_e \ran T_2$. 
 Hence, the equality holds.
\end{proof}

\noindent Next result exhibits a relation between the rank of\; $T\in L_1^{nm} \times L_1^{nm}$ with the ranks of its components.

\begin{theorem}\label{rang}
Let $T\in L_1^{nm} \times L_1^{nm}$ so that $T=e_1 T_1 + e_2 T_2$ for some $T_1,T_2\in L_1^{nm}$. Then, 
\begin{eqnarray*}
\dim \ran{T} = \dim \ran {T_1} + \dim \ran {T_2}. 
\end{eqnarray*}
In other words, that rank of $T$ is sum of the ranks of its own components.
\end{theorem}
\begin{proof} Let $\mathcal{B}_1  =\{z_i: i = =1,2,\ldots,p \}, \mathcal{B}_2 = \{w_j: j = 1,2,\ldots,q \}$ be the bases for $\ran{T_1}$ , $\ran{T_2}$ respectively.
Consider the collection $\mathcal{B}=\{e_1 z_i,e_2 w_j\}$ of \;$p+q$ \;elements of $\C_2$. Notice that $ <\mathcal{B}> \quad \subseteq \ran{T}$. If we take $(\xi_{1},\xi_{2},\ldots,\xi_{m}) \in\ran{T}$, then from theorem \ref{Kernel}, it follows that \begin{eqnarray*}
(\xi_{1}^-, \xi_{2}^-,\ldots,\xi_{m}^-) \in \ran T_{1} &\mbox{and}& (\xi_{1}^+, \xi_{2}^+,\ldots,\xi_{m}^+) \in \ran T_{2}\\
\Rightarrow \quad (\xi_{1}^-, \xi_{2}^-,\ldots,\xi_{m}^-) = \sum_{i=1}^{p} \alpha_i  z_i  &\mbox{and}& (\xi_{1}^+,\xi_{2}^+,\ldots,\xi_{m}^+) = \sum_{j=1}^{q} \beta_j  w_j , \quad \mbox{for some} \quad \alpha_i, \beta_j \in \C_1\\
\Rightarrow \quad (\xi_{1},\xi_{2},\ldots,\xi_{m}) &=& e_1 \sum_{i=1}^{p} \alpha_i  z_i + e_2 \sum_{j=1}^{q} \beta_j  w_j\\
 &=&  \sum_{i=1}^{p} \alpha_i(e_1  z_i) +  \sum_{j=1}^{q} \beta_j  (e_2 w_j).
\end{eqnarray*}
This implies $(\xi_{1},\xi_{2},\ldots,\xi_{m}) \in \; <\mathcal{B}>$. We thus have $\ran{T} = <\mathcal{B}>.$
Furthermore, if we consider
\begin{eqnarray*}
\sum_{i=1}^{p} \alpha_i (e_1 z_i) &+&  \sum_{j=1}^{q} \beta_j (e_2 w_j) = 0\\ 
\Rightarrow \quad e_1 \left(\sum_{i=1}^{p} \alpha_i z_i\right) &+& e_2 \left(\sum_{j=1}^{q} \beta_j w_j\right) = 0\\
\Rightarrow \quad  \sum_{i=1}^{p} \alpha_i z_i = 0 &\mbox{and}& \sum_{j=1}^{q} \beta_j w_j = 0 \\
\Rightarrow \quad \alpha_i = 0 = \beta_j, && \forall \  i , j.
\end{eqnarray*}
This forces $\mathcal{B}$ to be linearly independent and hence
$\mathcal{B}$ becomes a basis for $\ran{T}$. As we know $\rank T=\dim \ran{T}$, this implies
\begin{eqnarray*}
\rank {T} &=& \mbox{no. of elements of set}\ \mathcal{B} = p + q \\
\Rightarrow \quad \rank {T} &=& \dim \ran T_1 \ + \ \dim \ran T_2\\
&=& \rank \ T_1 \ + \ \rank \ T_2. 
\end{eqnarray*}
\end{proof}
\noindent Using rank-nullity theorem, one can deduce the same analogous result for nullity of linear map $T=e_1 T_1 + e_2 T_2$.. 
\begin{corollary}
For any $T=T=e_1 T_1 + e_2 T_2\in L_1^{nm} \times L_1^{nm}$, we have  
\begin{eqnarray*}
\dim \ker T &=& \dim \ker T_1 + \dim \ker T_2.
\end{eqnarray*}
Hence, in other words, nullity of $T$ is sum of nullity of its components.
\end{corollary}

\noindent {\bf Invertible and non-singular linear maps:} We now examine the invertibility and non-singularity of linear maps of the space $L_1^{nm}\times_e L_1^{nm}$. For article to be self-contained, these notions are first defined here.
\begin{definition}
Let $V,W$ be any two finite dimensional vector spaces over the same field. A linear map $T \colon V \to W$ is said to be 
\begin{itemize}
\item {\it invertible} if there exists a linear map $S \colon W \to V$ such that $S\circ T = I_{V}$ and $T\circ S = I_{W}$. In other words, $T$ is invertible if and only if  $T$ is bijective. 
\item {\it non-singular }if $\ker T = \{ 0 \}$, equivalently if $T$ is injective.
\end{itemize} 
\end{definition} 
\begin{remark}
For invertible linear map $T$,  the linear map $S$ given in the above definition is called the inverse of $T$.  In case the inverse exists it is unique, so we can denote it by $T^{-1}$. Also notice that an invertible linear map is an isomorphism so that the dimensions of $V$ and $W$ must be same, i.e., $\dim V=\dim W$.
\end{remark}
The following theorem asserts that idempotent product behaves well with respect to these notions.

\begin{theorem} \label{theorem3.6}
Let $T=e_{1} T_{1} + e_{2} T_{2}$ be an element of $ L_1^{nm}\times_e L_1^{nm}$. Then, we have
\begin{enumerate}
\item $T$ is a invertible if and only if both $T_{1},T_{2}$ are invertible. Further in this case, we have following 
\[
(e_{1} T_{1} + e_{2} T_{2})^{-1} = e_{1} T_{1}^{-1} + e_{2} T_{2}^{-1}.
\]
\item $T$ is non-singular if and only if both $T_{1},T_{2}$ are non-singular.
\end{enumerate}
\end{theorem}
\begin{proof} Suppose $T=e_{1} T_{1} + e_{2} T_{2}$ is invertible. Then there exists a linear map \ $S=e_{1} S_{1} + e_{2} S_{2}$ such that
\begin{eqnarray*}
S\circ T = T\circ S &=&  I_{\C_{2}^{n}} \\
\Leftrightarrow \quad S_{1}\circ T_{1} = T_{1}\circ S_{1} = I_{\C_{1}^{n}} &\mbox{and}&  S_{2}\circ T_{2} = T_{2}\circ S_{2} =I_{\C_{1}^{n}} \\
\Leftrightarrow \quad S_{1} \ \mbox{is the inverse of} \ T_{1} \ &\mbox{and}& \ S_{2} \ \mbox{is the inverse of} \ T_{2}\\
\Leftrightarrow \quad T_{1}\;\mbox{and}\; T_{2}\; \mbox{are invertible}.&&
\end{eqnarray*}
Clearly, $S=T^{-1}$ if and only if $S_1=T_1^{-1}$ and $S_2=T_2^{-1}$. This implies that $(e_{1} T_{1} + e_{2} T_{2})^{-1} = e_{1} T_{1}^{-1} + e_{2} T_{2}^{-1}.$  Now for the second part, we first assume that $e_{1} T_{1} + e_{2} T_{2}$ is non-singular. Then, by using theorem \ref{Kernel}, we have
\begin{eqnarray*}
&\Leftrightarrow& \ker(e_{1} T_{1} + e_{2} T_{2}) = \{0 \}\\
&\Leftrightarrow& \ker(T_{1}) = \{0 \},\; \ker(T_{2}) = \{0 \}\\
&\Leftrightarrow& T_{1}, T_{2} \;\; \mbox{are non-singular}.
\end{eqnarray*}
This completes the proof.
\end{proof}
\noindent The following theorem generalizes the result of $C_1$-linear maps to the elements of $L_1^{nm}\times_e L_1^{nm}$. 
\begin{theorem} \label{theorem3.8}
 Let $T=e_{1} T_{1} + e_{2} T_{2} \colon \C_2^{n} \to \C_2^{m},\  S=e_{1} S_{1} + e_{2} S_{2} \colon \C_2^{m} \to \C_2^{k}$ be the linear maps. Suppose that ${\mathcal{B}_{1}}, {\mathcal{B}_{2}}$ and ${\mathcal{B}_{3}}$ are the ordered bases for $\C_1^{n}, \C_1^{m}$ and $\C_1^{k}$ respectively. Then, we have
 \[
 \left[ S\circ T \right]^{\mathcal{B}_{1}}_{\mathcal{B}_{3}} = 
 [S]^{\mathcal{B}_{2}}_{\mathcal{B}_{3}}\cdot[T]^{\mathcal{B}_{1}}_{\mathcal{B}_{2}}.
 \]
 \end{theorem}
\begin{proof}
By using theorem \ref{properties} and definition \ref{methodl}, it follows
\begin{eqnarray*}
\left[ S\circ T \right]^{\mathcal{B}_{1}}_{\mathcal{B}_{3}} &=& [(e_{1} S_{1} + e_{2} S_{2}) \circ (e_{1} T_{1} + e_{2} T_{2})]^{\mathcal{B}_{1}}_{\mathcal{B}_{3}}\\
&=& [e_{1} (S_{1}\circ T_{1}) + e_{2} (S_{2} \circ T_{2})]^{\mathcal{B}_{1}}_{\mathcal{B}_{3}}  \\
 &=& e_{1} [S_{1} \circ T_{1}]^{\mathcal{B}_{1}}_{\mathcal{B}_{3}} + e_{2} [S_{2}\circ T_{2}]^{\mathcal{B}_{1}}_{\mathcal{B}_{3}}  \\
 &=& e_{1} [S_{1}]^{\mathcal{B}_{2}}_{\mathcal{B}_{3}} \cdot  [T_{1}]^{\mathcal{B}_{1}}_{\mathcal{B}_{2}} + e_{2} [S_{2}]^{\mathcal{B}_{2}}_{\mathcal{B}_{3}} \cdot [T_{2}]^{\mathcal{B}_{1}}_{\mathcal{B}_{3}} \\
&=&  (e_{1} [S_{1}]^{\mathcal{B}_{2}}_{\mathcal{B}_{3}} + e_{2} [S_{2}]^{\mathcal{B}_{2}}_{\mathcal{B}_{3}})  \cdot (e_{1} [T_{1}]^{\mathcal{B}_{1}}_{\mathcal{B}_{2}} + e_{2} [T_{2}]^{\mathcal{B}_{1}}_{\mathcal{B}_{2}})   \\
&=& [e_{1}  S_{1} + e_{2} S_{2}]^{\mathcal{B}_{2}}_{\mathcal{B}_{3}}  \ \cdot [e_{1}T_{1} + e_{2} T_{2}]^{\mathcal{B}_{1}}_{\mathcal{B}_{2}} \\ 
&=& [S]^{\mathcal{B}_{2}}_{\mathcal{B}_{3}}\cdot[T]^{\mathcal{B}_{1}}_{\mathcal{B}_{2}}.
\end{eqnarray*}
This completes the theorem.
\end{proof}

\begin{remark}\label{specialcase2} As a special case in above theorem, if we take $\C_2^{m} = \C_2^{n} = \C_2^{k}$ and $\mathcal B$ as a basis, then we get
\begin{eqnarray}
[S \circ T]_{\mathcal{B}} = [S]_{\mathcal{B}} \cdot [T]_{\mathcal{B}}. 
\end{eqnarray}
\end{remark}

\noindent {\bf Invertible and non-singular bicomplex matrices:} We now discuss the invertibility and non-singularity of bicomplex matrix of the space $\C_2^{n\times n}$. First we define them here.
\begin{definition}
A bicomplex square matrix $A\in\C_2^{n\times n}$ is said to be 
\begin{itemize}
\item {\it invertible} if there exists a matrix $B \in \C_2^{n\times n}$ such that $A\cdot B = B\cdot A=I_n$. 
\item {\it non-singular} if $\det A$ is a non-singular element of $\C_2$. 
\end{itemize} 
\end{definition}  
\begin{remark}\label{remark3} Every square matrix $A\in \C_2^{n\times n}$, like bicomplex number, can also be written uniquely as $A=e_1A_1+e_2A_2$, where $A_1,A_2\in \C_1^{n\times n}$ are complex square matrices. With this representation of bicomplex matrices, we have (cf. \cite[exercise 6.8]{price2018})
\begin{enumerate}
\item $A$ is invertible if and only if $A_1,A_2$ are invertible.
\item $A$ is non-singular if and only if $A_1,A_2$ are non-singular.
\end{enumerate}
\end{remark}

\begin{theorem}
Let $T\in L_1^{nm}\times_e L_1^{nm}$ and $\mathcal B$ be a basis for $\C_1^n$. Then, the linear map $T$ is invertible if and only if the matrix\; $[T]_{\mathcal{B}}$\; is invertible.
\end{theorem}
\begin{proof}
Let $T = e_{1} T_{1} + e_{2}T_2$ is invertible. This implies that $T^{-1}$ exists. From Part (1) of theorem \ref{theorem3.6}, it follows that \;$T^{-1} =  e_{1} T_{1}^{-1} + e_{2} T_{2}^{-1}$.  We thus have
\begin{eqnarray*}
T \circ T^{-1} &=&  I_{\C_{2}^{n}}\\
\therefore \quad [T \circ T^{-1}]_\mathcal{B} &=&  I_n\\
{[T]}_{\mathcal B} \cdot [T^{-1}]_{\mathcal B} &=&  I_n \quad\mbox{(from remark \ref{specialcase2})}.
\end{eqnarray*}
Hence, $[T]_\mathcal{B}$ is invertible. \\
Conversely suppose $[T]_\mathcal{B}$ is invertible. By using part (1) of remark \ref{remark3},  we have  
\begin{eqnarray*}
&&[T_1]_\mathcal{B} , [T_2]_\mathcal{B} \ \mbox{are invertible}\\
&\Rightarrow& \quad T_1 , T_2 \  \mbox{are invertible}\\
&\Rightarrow& \quad T \ \mbox{is invertible}.
\end{eqnarray*}
This completes the theorem.
\end{proof}

\bibliographystyle{amsplain}
\bibliography{references}

\providecommand{\bysame}{\leavevmode\hbox to3em{\hrulefill}\thinspace}
\providecommand{\MR}{\relax\ifhmode\unskip\space\fi MR }
\providecommand{\MRhref}[2]{%
  \href{http://www.ams.org/mathscinet-getitem?mr=#1}{#2}
}
\providecommand{\href}[2]{#2}
\begin{thebibliography}{10}

\bibitem{alpay2023interpolation}
Daniel Alpay, Izchak Lewkowicz, and Mihaela Vajiac, \emph{Interpolation with symmetry and a herglotz theorem in the bicomplex setting}, Journal of Mathematical Analysis and Applications \textbf{524} (2023), no.~2, 127201.

\bibitem{alpay2014basics}
Daniel Alpay, Maria~Elena Luna-Elizarrar{\'a}s, Michael Shapiro, and Daniele~C Struppa, \emph{Basics of functional analysis with bicomplex scalars, and bicomplex schur analysis}, Springer Science \& Business Media, 2014.

\bibitem{cerroni2017theory}
Cinzia Cerroni, \emph{From the theory of “congeneric surd equations” to “segre's bicomplex numbers”}, Historia Mathematica \textbf{44} (2017), no.~3, 232--251.

\bibitem{colombo2014bicomplex}
Fabrizio Colombo, Irene Sabadini, and Daniele~C Struppa, \emph{Bicomplex holomorphic functional calculus}, Mathematische Nachrichten \textbf{287} (2014), no.~10, 1093--1105.

\bibitem{futagawa1928}
Michiji Futagawa, \emph{On the theory of functions of a quaternary variable}, Tohoku Mathematical Journal, First Series \textbf{29} (1928), 175--222.

\bibitem{futagawa1932}
\bysame, \emph{On the theory of functions of a quaternary variable (part ii)}, Tohoku Mathematical Journal, First Series \textbf{35} (1932), 69--120.

\bibitem{kumar2015bicomplex}
Romesh Kumar and Kulbir Singh, \emph{Bicomplex linear operators on bicomplex hilbert spaces and littlewood’s subordination theorem}, Advances in Applied Clifford Algebras \textbf{25} (2015), no.~3, 591--610.

\bibitem{price2018}
G~Baley Price, \emph{An introduction to multicomplex spaces and functions}, CRC Press, 2018.

\bibitem{riley1953}
James~D Riley, \emph{Contributions to the theory of functions of a bicomplex variable}, Tohoku Mathematical Journal, Second Series \textbf{5} (1953), no.~2, 132--165.

\bibitem{ringleb1933}
Friedrich Ringleb, \emph{Beitr{\"a}ge zur funktionentheorie in hyperkomplexen systemen i}, Rendiconti del Circolo Matematico di Palermo (1884-1940) \textbf{57} (1933), no.~1, 311--340.

\bibitem{rochon2004}
Dominic Rochon and Michael Shapiro, \emph{On algebraic properties of bicomplex and hyperbolic numbers}, Anal. Univ. Oradea, fasc. math \textbf{11} (2004), no.~71, 110.

\bibitem{srivastava2003}
Rajiv~K Srivastava, \emph{Bicomplex numbers: analysis and applications}, Mathematics Student-India \textbf{72} (2003), no.~1-4, 69--88.

\bibitem{srivastava2008}
\bysame, \emph{Certain topological aspects of bicomplex space}, Bull. Pure \& Appl. Math \textbf{2} (2008), no.~2, 222--234.

\bibitem{srivastava2011}
Rajiv~K Srivastava and Sukhdev Singh, \emph{On bicomplex nets and their confinements}, Amer. J. of Math. and Stat \textbf{1} (2011), no.~1, 8--16.

\end{thebibliography}
\end{document}